\documentclass[10pt,a4paper]{article}
\usepackage{mathrsfs}
\usepackage{amsfonts}
\usepackage[active]{srcltx}
\usepackage{indentfirst,latexsym,bm,amsmath,pstricks,amssymb,amsthm,amscd}
\usepackage[all]{xy}
\begin{document}
\theoremstyle{plain}
\newtheorem{thm}{Theorem}[section]
\newtheorem{theorem}[thm]{Theorem}
\newtheorem{addendum}[thm]{Addendum}
\newtheorem{lemma}[thm]{Lemma}
\newtheorem{corollary}[thm]{Corollary}
\newtheorem{proposition}[thm]{Proposition}

\newcommand{\mR}{\mathbb{R}}
\newcommand{\mZ}{\mathbb{Z}}
\newcommand{\mN}{\mathbb{N}}
\newcommand{\mC}{\mathbb{C}}
\newcommand{\mP}{\mathbb{P}}
\newcommand{\mQ}{\mathbb{Q}}
\newcommand{\CO}{\mathcal{O}}
\newcommand{\CE}{\mathcal{E}}
\newcommand{\CF}{\mathcal{F}}
\newcommand{\CG}{\mathcal{G}}
\newcommand{\CL}{\mathcal{L}}
\newcommand{\CM}{\mathcal{M}}
\newcommand{\CP}{\mathcal{P}}
\newcommand{\CS}{\mathcal{S}}
\newcommand{\CA}{\mathcal{A}}
\newcommand{\CB}{\mathcal{B}}
\newcommand{\CC}{\mathcal{C}}
\newcommand{\CH}{\mathcal{H}}
\newcommand{\CI}{\mathcal{I}}
\newcommand{\CJ}{\mathcal{J}}
\newcommand{\CN}{\mathcal{N}}
\newcommand{\CR}{\mathcal{R}}
\newcommand{\CZ}{\mathcal{Z}}

\newcommand{\mr}{\mbox}
\newcommand{\I}{{\bf I}}
\newcommand{\J}{{\bf J}}
\newcommand{\m}{{\bf m}}
\newcommand{\ord}{\mathrm{ord}}

\renewcommand{\thefootnote}{\fnsymbol{footnote}}

\theoremstyle{definition}
\newtheorem{remark}[thm]{Remark}
\newtheorem{remarks}[thm]{Remarks}
\newtheorem{notations}[thm]{Notations}
\newtheorem{definition}[thm]{Definition}
\newtheorem{claim}[thm]{Claim}
\newtheorem{assumption}[thm]{Assumption}
\newtheorem{assumptions}[thm]{Assumptions}
\newtheorem{property}[thm]{Property}
\newtheorem{properties}[thm]{Properties}
\newtheorem{example}[thm]{Example}
\newtheorem{examples}{Examples}
\newtheorem{conjecture}[thm]{Conjecture}
\newtheorem{questions}[thm]{Questions}
\newtheorem{question}[thm]{Question}
\newtheorem{problem}[thm]{Problem}
\numberwithin{equation}{section}
 \newcommand{\Rnm}[1]{\uppercase\expandafter{\romannumeral #1}}

\title{Modular invariants and singularity indices of hyperelliptic fibrations}
\author{Xiao-Lei Liu*}

\footnotetext[1]{Academy of Mathematics and Systems Science, Chinese Academy of
Science, Beijing, P.R. China.

 E-mail: xlliu1124@amss.ac.cn}

%

\date{}

 \maketitle

\textbf{Abstract} {The modular invariants of a family of semistable curves
are the degrees of the corresponding divisors on the image of
the moduli map. The singularity indices were introduced by G. Xiao
to classify singular fibers of hyperelliptic fibrations and to
compute global invariants locally. In the semistable case, we show
that the modular invariants corresponding with the boundary classes
are just the singularity indices. As an application, we show that
the formula of Xiao for relative Chern numbers is the same as that
of Cornalba-Harris in the semistable case. }


{\textbf{2000 MR Subject Classification}
14D06, 14D22, 14H10 }

\section{Introduction}
The modular invariants of a family of curves were introduced by Tan (\cite{Ta10}). They are the degrees of the
corresponding divisors  on the image of the moduli map. 
In the language of arithmetic algebraic geometry, a modular
invariant is a certain height of arithmetic curves, for example,
Faltings height is the modular invariant corresponding to Hodge
class. Modular invariants can be used to describe the lower bound
for effective Bogomolov conjecture which is about the finiteness of
algebraic points of small height (\cite{Zh93,Zh10}). More recently,
Prof. Tan found that the modular invariants are invariants
of differential equations, which were expected by mathematicians in
19th century to study the qualitative properties of differential
equations (\cite{Ta}).

Historically, the study of fibred surfaces is started by Kodaira
(\cite{Kor}), who gave a complete classification theory for elliptic
fibrations. This combinatoric classification of elliptic fibers
is used in the computation of the modular invariants.  But such a classification is too complicate for the case when the genus $g\geq2$.  There are more than one hundred classes of singular
fibers of genus $2$ (\cite{NU,Ogg}), and the number of classes of
singular fibers increases quickly as the genus becomes bigger. Horikawa
(\cite{Hor}) classified the singular fibers of genus $g=2$
into 5 classes from a different point of view. Based on Horikawa's work, Xiao
(\cite{Xi91,Xi92}) introduced the singularity indices (see
Definition \ref{defnsingularindex}) to classify singular fibers for
hyperelliptic fibrations, furthermore, he obtained the local-global
formulas, and determined the fundamental group from his
classification.

In what follows, we will prove that these two basic invariants, the
modular invariants corresponding to boundary classes and the
singularity indices, coincide with each other for semistable
fibrations.

Before starting this result, we explain our notations and
assumptions.

 A family of curves
of genus $g$ is a fibration $f:S\to C$ whose general fibers $F$ are
smooth curves of genus $g$, where $S$ is a complex smooth projective
surface, and $C$ is a smooth curve of genus $b$. The family is
called {\it semistable} if all the singular fibers are semistable
curves. (Recall that a semistable curve $F$ is a reduced connected
curve that has only nodes as singularities and every smooth rational
components of $F$ meets the other components at no less than 2
points.) If all the smooth fibers are hyperelliptic, we say that the
family is {\it hyperelliptic}. We always assume that $f$ is {\it
relatively minimal}, i.e., there is no $(-1)$-curve in any singular
fiber.

If $r$ is a non-negative real number, we denote by $[r]$ the integral
part of $r$. Hence when $m$ is a positive integer, $m-2[m/2]$ is
zero if $m$ is even, or 1 otherwise.

For a fibration $f:S\to C$, we have three fundamental relative
invariants which are non-negative,
\begin{equation}
\begin{split}
K_f^2&=K_{S/C}^2=K_S^2-8(g-1)(b-1), \\
e_f&=\chi_{\mathrm{top}}(S)-4(g-1)(b-1),\\
\chi_f&=\deg f_*\omega_{S/C}=\chi(\mathcal O_S)-(g-1)(b-1).
\end{split}
\end{equation}

Let $f$ be a locally non-trivial fibration, the slope of $f$ is
defined as
$$
\lambda_f=K_f^2/\chi_f.
$$

For $g\geq2$, let the moduli map induced by a semistable family $f$
be
$$J: C\to \overline{\CM}_g,$$
which is a holomorphic map from $C$ to the moduli space
$\overline{\CM}_g$ of semistable curves of genus $g$. For each
$\mathbb{Q}$-divisor class $\eta$ of the moduli space
$\overline{\CM}_g$, we can define an invariant $\eta(f)=\deg
J^*\eta$ which satisfies the {\it base change property}, i.e., if
$\tilde f:\tilde X\to \tilde C$ is the pullback fibration of $f$
under a base change $\pi:\tilde C \to C$ of degree $d$, then
$\eta(\tilde f)=d\cdot \eta(f)$ (see \cite{Ta10}).  Consequently,
for a non-semistable family $f$, we have
\begin{equation}
\eta(f)=\frac{\eta(\tilde f)}{d},
\end{equation}
 where $\tilde f$ is the
semistable model of $f$ corresponding to a base change of degree
$d$.

We call the invariant $\eta(f)$ of the family $f$ {\it the modular
invariant corresponding to $\eta$}.

Let $\Delta_0,\ldots,\Delta_{[g/2]}$ be the boundary divisors of
$\overline{\CM}_g$, and $\delta_i(f)$ be the modular invariant
corresponding to the divisor class $\delta_i=[\Delta_i]$ in
$\mbox{Pic}(\overline{\CM}_g)\otimes\mathbb{Q}$,
$i=0,1,\ldots,[g/2]$. Let
$\lambda\in\mbox{Pic}(\overline{\CM}_g)\otimes\mathbb{Q}$ be the
Hodge class, $\delta=\delta_0+\ldots+\delta_{[g/2]}$, and
$\kappa=12\lambda-\delta$. For these classes, we have the {\it
modular invariants}
 $\lambda(f)$,  $\delta(f)$ and $\kappa(f)$ of $f$. If $f$ is
 semistable, then
\begin{equation}\label{modrel}
\lambda(f)=\chi_f,~~ \delta(f)=e_f,~~ \kappa(f)=K_f^2.
\end{equation}

We say that a singularity $p$ in a semistable curve $F$ is a  {\it
node of type} $i$ if its partial normalization at $p$ consists of
two connected components of arithmetic genera $i$ and $g-i\geq i$,
for $i>0$, and is connected for $i=0$. The node of the semistable
curve corresponding to a general point of $\Delta_0$ is {\it
$\alpha$-type}, i.e., an ordinary double point of an irreducible
curve, hence it is a node of type $0$. For a general point in
$\Delta_i$, the corresponding node is of type $i~(i\geq1)$ (see the
following figure).

 \setlength{\unitlength}{1mm}
\begin{center}
\begin{picture}(40,20)(0,-5)
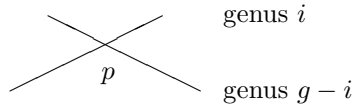
\put(-8,-5){\makebox(-10,0)[l]{$~\mbox{Figure 1: Node of type~} i~
(i\geq1)$}} \put(0,0){\line(2,1){20}} \put(25,0){\line(-2,1){20}}
\put(12,2){\makebox(0,0)[l]{$p$}}
\put(28,10){\makebox(0,0)[l]{\mbox{genus}~$i$}}
\put(28,0){\makebox(0,0)[l]{\mbox{genus}~$g-i$}}
\end{picture}
\end{center}

Denote by $\delta_i(F)$ the number of nodes of type $i~(i\geq0)$.

The general point in the intersection
$\Delta_{i_1}\cap\cdots\cap\Delta_{i_k}$ of $k$ distinct boundary
divisors  corresponds to a semistable curve with $k$ nodes which are
of types $i_1,\ldots,i_k$ respectively.

For the moduli space $\overline{\CH}_g$ of semistable hyperelliptic
curves, the intersection of $\Delta_0$ with $\overline{\CH}_g$
breaks up into $\Xi_0, \Xi_1,\ldots,\Xi_{[(g-1)/2]}$. We denote by
$\Theta_i$ the restriction of $\Delta_i~(i\geq1)$ on
$\overline{\CH}_g$. Suppose $F$ is  a semistable hyperelliptic curve
with hyperelliptic involution $\sigma$, and $p\in F$ is a  node of
type 0. If $p=\sigma(p)$, then we set $k=0$; if $p\neq \sigma(p)$,
and the partial normalization of $F$ at $p$ and $\sigma(p)$ consists
of two connected components of arithmetic genera $k$ and $g-k-1\geq
k$, then the node $p$ (resp. nodal pair $\{p,\sigma(p)\}$) is called
a {\it node} (resp. {\it nodal pair) of type} $(0,k)$. Then the
nodes of semistable curves corresponding to a general point of
$\Xi_k$ are of type $(0,k)$ (see the following figure).

 \setlength{\unitlength}{1mm}
\begin{center}
\begin{picture}(40,15)(0,0)
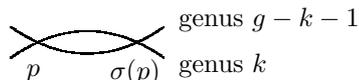
\put(-8,0){\makebox(-10,0)[l]{$\mbox{Figure 2: Nodes of type
~}(0,{k})~~(k\geq0)$}} \qbezier(0,8)(10,15)(20,8)
\qbezier(0,12)(10,5)(20,12) \put(2,6){\makebox(0,0)[l]{$p$}}
\put(13,6){\makebox(0,0)[l]{$\sigma(p)$}}
\put(22,7){\makebox(0,0)[l]{\mbox{genus}~$k$}}
\put(22,13){\makebox(0,0)[l]{\mbox{genus}~$g-k-1$}}
\end{picture}
\end{center}

A semistable hyperelliptic curve is a double cover of a tree of
rational curves branched over $2g+2$ points (see X.3 in \cite{ACG}),
which is induced by the involution map.  Since the points $p$ and
$\sigma(p)$ map to the same point in some $\mP^1$, we treat them
together as a nodal pair $\{p,\sigma(p)\}$.

Let
\begin{equation*}
\begin{split}
&\CN_{2,1}(F)=\{p\in F:~p\mbox{ is a node of type }(0,0),~p=\sigma(p)\},\\
&\CN_{2,2}(F)=\{\{p,\sigma(p)\}\subset F:~\{p,\sigma(p)\}\mbox{ is a
nodal pair of type }(0,0),~p\neq\sigma(p)\}.
\end{split}
\end{equation*}
Denote by $\CN_{2k+2}(F)$ (resp. $\CN_{2k+1}(F)$) the set of all the
nodal pairs $\{p,\sigma(p)\}$ of type $(0,k)$ (resp. nodes $p$ of
type $k$) ($k>0$). Then we define
\begin{equation}\label{xi0eqn}
\begin{split}
\xi_0(F)&:=|\CN_{2,1}(F)|+2|\CN_{2,2}(F)|,\\
\xi_k(F)&:=|\CN_{2k+2}(F)|,~~\delta_k(F):=|\CN_{2k+1}(F)|,~k\geq1.
\end{split}
\end{equation}

From now on, we assume that $f:S\to C$ is hyperelliptic. Suppose $f$
is semistable, let $\delta_k(f)$ (resp. $\xi_k(f)$) be the modular
invariants corresponding to the boundary divisors $\Theta_k$ (resp.
$\Xi_k$). Then (cf. \cite{CH88})
\begin{equation}
\delta_k(f)=\sum_{i=1}^s\delta_k(F_i)~~(k\geq1),~~\xi_k(f)=\sum_{i=1}^s\xi_k(F_i)~~(k\geq0),
\end{equation}
where $F_1,\ldots,F_s$ are all singular fibers of $f$, and
\begin{equation}\label{delta0eqn}
\delta_0(f)=\sum_{k\geq0} \xi_k(f).
\end{equation}
It's proved that in \cite{CH88}, if $f$ is a semistable fibration,
then
\begin{equation}\label{CHeq}
\begin{split}
(8g+4)\lambda(f)=&g\xi_0(f)+\sum_{k=1}^{[{(g-1)}/2]}2(k+1)(g-k)\xi_k(f)+\sum_{k=1}^{[g/2]}4k(g-k)\delta_k(f),\\
\delta(f)=&\xi_0(f)+\sum_{k=1}^{[(g-1)/2]}2\xi_k(f)+\sum_{k=1}^{[g/2]}\delta_k(f).
\end{split}
\end{equation}

On the other hand, for a hyperelliptic fibration $f:S\to C$, the
relative canonical map $\Phi: S\dashrightarrow
\mbox{Proj}(f_*\omega_{S/C})$ induced by $f_*\omega_{S/C}$ is a
generic double cover. Then we can choose a reasonable double cover
which is determined by genus $g$ datum $(P,R,\delta)$, where $P$ is
a geometric ruled surface $\varphi: P\to C$, $R$ is the branch
locus, and $\delta$ is the square root of $R$ (see Section
\ref{sectgenusgdata}). Thus there is a map $\varphi_R:R\to C$
induced by $\varphi$. Xiao introduced the singularity indices
$s_2(f),s_3(f),\ldots, s_{g+2}(f)$ (see Definition
\ref{defnsingularindex}),  to describe the contribution of the
singular points of $R$, the smooth ramified points of $\varphi_R$
and the vertical components of $R$ to the relative invariants
$K_f^2,~\chi_f$ and $e_f$. He obtained the following local-global
formulas using these singularity indices $s_k(f)$'s (see Theorem
\ref{Theorem 5.1.7}),
\begin{equation}\label{xiaoeq}
\begin{split}
(8g+4)\chi_f=&g\big(s_2(f)-2s_{g+2}(f)\big)+\sum_{k=2}^{[\frac{g-1}2]}2(k+1)(g-k)s_{2k+2}(f)\\
&+\sum_{k=1}^{[\frac{g+1}2]}4k(g-k)s_{2k+1}(f),\\
e_f=&
s_2(f)-3s_{g+2}(f)+\sum_{k=1}^{[\frac{g-1}2]}2s_{2k+2}(f)+\sum_{k=1}^{[\frac{g+1}2]}s_{2k+1}(f).
\end{split}
\end{equation}
Note that Xiao's equations do not need the semistable condition.

If $f$ is semistable, then $s_{g+2}(f)=0$ (see Corollary
\ref{g+2=0}). Comparing equations (\ref{CHeq}) with (\ref{xiaoeq}),
it is natural to build up the relation between modular invariants
with singularity indices.

A double point $p$ of a semistable curve $F$ is called {\it
separable} if $F$ becomes disconnected when normalize $F$ locally at
$p$; otherwise, $p$ is called {\it inseparable}. Xiao showed that
for each semistable fibration $f$ of genus 2, $s_2(f)$ (resp.
$s_3(f)$) is the number of inseparable (resp. separable) double
points of all singular fibers of $f$ (\cite{Xi92}), i.e.,
\begin{equation}
\xi_0(f)=s_2(f),~~~\delta_1(f)=s_3(f).
\end{equation}

If we subdivide the inseparable nodal points into nodes of type
$(0,k)~(k\geq0)$, and subdivide the separable nodes into nodes of
type $i~(i\geq1)$, then we can  get that the modular invariants
$\delta_i(f)$,
 $\xi_j(f)$ are the same as the singularity indices $s_k(f)$ for each $g\geq2$:
 \begin{theorem}\label{mainthm}
 Suppose $f$ is a semistable hyperelliptic fibration of genus $g$, then
\begin{equation}\label{modsing}
\delta_k(f)=s_{2k+1}(f)~(k\geq1),~~ \xi_k(f)=s_{2k+2}(f)~(k\geq0).
\end{equation}
Hence
\begin{equation}
\delta_0(f)=s_2(f)+2s_4(f)+\cdots+2s_{2[(g-1)/2]+2}(f).
\end{equation}
\end{theorem}

Considering the equations in (\ref{modrel}) and (\ref{modsing}), it
is likely that there exists a more general correspondence between
modular invariants and relative invariants. Precisely, we expect
that if $\CM$ is any kind of moduli space, and $\eta$ is a divisor
class of $\CM$, especially the generator of $\mathrm{Pic}(\CM)$,
there is a reasonable relative invariant which coincides with the
modular invariant $\eta(f)$ corresponding to $\eta$ for each
semistable family $f$ of curves in $\CM$. Recently, there is another
such corresponding showed in \cite{CLT}.

In $\S2$, we recall Xiao's study of hyperelliptic fibration,
including the reason for starting from genus $g$ datum, the
definition of singularity indices, and the local-global formulas. In
$\S3$, we repeat the work \cite{Tu08} of Yuping Tu on semistable
criterion firstly, which concerns the sufficient and necessary
conditions of branch locus such that the fibration is semistable.
From these conditions, we prove our result locally by constructing
bijective maps between sets of singularities $\CR_*$ with sets of
nodes (or nodal pairs) $\CN_*$.

\section{Singularity indices}\label{sectsingindex}
\subsection{Genus $g$ data}\label{sectgenusgdata}
For the reader's convenience, we recall the notions of double cover
and minimal even resolution firstly.

Let $P$ be a smooth surface, and $R$ a {\it reduced even divisor}
(the image of $R$ in $\mathrm{Pic}(P)$ is divisible by 2) on $P$.
Let $\delta$ be an invertible sheaf such that
$\CO_P(R)=\delta^{\otimes2}$, and we call $\delta$ the {\it square
root} of $R$ for convenience. In fact,  a reduced even divisor $R$
on $P$ and an invertible sheaf $\delta$ with
$\CO_P(R)=\delta^{\otimes2}$ determine a unique double cover
$\pi:S\to P$ branched along $R$ (see I.7 in \cite{BPV}). Thus $(R,
\delta)$ is called a {\it double cover datum}. If $R$ is reduced
smooth, then $S$ is smooth.

If $\psi_1:P_1\to P$ is a blowing-up of $P$ centered at a point $x$
of $R$ of order $m$, set
\begin{equation}
R_1:=\psi^*_1(R)-2[m/2]E,~~~ \delta_1:=\psi^*_1\delta-[m/2]E,
\end{equation}
where $E$ is the exceptional $(-1)$-curve of $\psi_1$. Then
$(R_1,\delta_1)$ is called a {\it reduced even inverse image} of
$(R,\delta)$ under $\psi_1$. In what follows, we call $R_1$ a
reduced even inverse image of $R$ briefly, since $\delta_1$ is
determined by $(R,\delta)$ and $R_1$.

\begin{definition}
An {\it even resolution} of $R$ is a sequence of blowing-ups
$\tilde\psi=\psi_1\circ\psi_2\circ\cdots\circ\psi_r:\tilde P\to P$
\begin{equation}\label{resolution}
\tilde\psi:(\tilde P,\tilde R){=}(P_r,R_r)\stackrel{\psi_r}{\to}
\cdots{\to} (P_2,R_2)\stackrel{\psi_2}{\to} (P_1,R_1)
\stackrel{\psi_1}{\to} (P_0,R_0){=}(P,R),
\end{equation}
satisfying the following conditions:

(i). $\tilde R$ is a smooth reduced even divisor,

(ii). $R_i$ is the reduced even inverse image of $R_{i-1}$ under
$\psi_i$.

Furthermore, $\tilde\psi$ is called the {\it minimal even
resolution} of the singularities of $R$ if

(iii). $\psi_i$ is the blowing-up of $P_{i-1}$ centered at a
singular point $x_i$ of $R_{i-1}$ for any $1\leq i\leq r$.
\end{definition}

If the even resolution of $\tilde\psi:\tilde P\to P$ of $R$ is
minimal, then for any even resolution $\psi':P'\to P$, there exists
a morphism $\alpha:P'\to \tilde P$ such that
$$\alpha(R')=\tilde R,~~\alpha(\delta')=\tilde \delta.$$
Here $\alpha(\delta')=\tilde \delta$ means that there exists a
divisor $D'\in \mathrm{Pic}(P')$ with $\delta'\cong\CO_{P'}(D')$
such that $\tilde \delta\cong\CO_{\tilde P}(\alpha(D'))$.

Note that the minimal even resolution is unique.

If $x_i\in P_{i-1}$ lies in $E_j~(j<i)$, that is,
$\psi_j\circ\cdots\circ\psi_{i-1}(x_i)=x_j$, we say that $x_i$ is
{\it infinitely near} $x_j$.

Let $x_i$ be a singularity of $R$ of order
$\mathrm{ord}_{x_i}(R)=m_i$. If $m_i\leq3$ and for any $x_j$
infinitely near $x_i$ $(j>i)$ we have $m_j\leq3$, then $x_i$ is
called a {\it negligible singularity}, since such a singularity does
not change the invariants $K_f^2,\chi_f$ (see (2) in \cite{Xi91}).

Unless stated otherwise, the singularities (resp. the smooth points)
of $R$ include all the infinitely near singularities (resp. the
smooth points) of $R_i$ in $P_i$ for $1\leq i\leq r$. If we want to
specify a singularity (resp. a smooth point) $p$ of $R$, we will
point out the surface which $p$ lies in.

Now we want to introduce the genus $g$ datum associated to a
hyperelliptic fibration $f:S\to C$, according to Xiao's approach in
\cite{Xi91,Xi92}.

Since the generic fiber $F$ of $f$ is hyperelliptic, we glue the
involution $\sigma_F$ of $F$ together, and then we get a rational
map $\sigma:S\to S$. The map $\sigma$ is in fact a morphism, because
$f$ is assumed to be relatively minimal. Let $\rho:\tilde S\to S$ be
the minimal composition of blowing-ups of $S$ centered at all the
isolated fixed points of $\sigma$, and $\tilde\sigma:\tilde S\to
\tilde S$ be the induced map of $\sigma$ on $\tilde S$. Then $\tilde
P=\tilde S/\langle\tilde\sigma\rangle$ is smooth. Let
$\tilde\theta:\tilde S\to \tilde P$ be the corresponding double
cover branched along a smooth reduced divisor $\tilde R$ in $\tilde
S$. Then $\tilde\theta_*(\CO_{\tilde S})\cong\CO_{\tilde P}\oplus
\tilde\delta^\vee$ where $\tilde\delta^\vee$ is an invertible sheaf
with $\tilde\delta^{\otimes2}\cong\CO_{\tilde S}(\tilde R)$.

Let $\Phi_K:S\dashrightarrow \mathrm{Proj}(f_*\omega_{S/C})$ be the
relative canonical map. $\Phi_K$ is a generic double cover, for its
restriction on a generic fiber $F$ of $f$ is the double cover
induced by the involution of $F$. Let $\hat\rho:\hat S\to S$ be the
minimal composition of blowing-ups centered at all base points of
$\Phi_K$ and all isolated fixed points. Then the birational morphism
$\hat S\to \tilde S$ is an isomorphism because of the minimality of
$\rho$. Hence $\hat\rho=\rho$ and $\hat S\cong\tilde S$. This gives
another process to get the double cover $\tilde\theta:\tilde S\to
\tilde P$ and the branch locus $\tilde R$.

\centerline{\mbox{\xymatrix{
\tilde{S} \ar@{->}"1,3"^-{\tilde{\theta}}\ar[d]_-{\rho} &  & {\tilde{P}} \ar[d]^-{\tilde\psi} \\
S  \ar@{-->}"2,3"^-{}  \ar[dr]_-f   &  & P \ar[dl]^-{\varphi}\\
& C & }}}

The morphism $\tilde \varphi:\tilde P\to C$ induced by $f$ is a
birational ruling (a fibration whose general fibers are rational
curves).  There are many choices to give a birational morphism
$\tilde\psi:\tilde P\to P$ from $\tilde P$ to a geometric ruled
surface $\varphi: P\to C$ over $C$ which induces a reduced divisor
$R=\tilde\psi(\tilde R)$ in $P$. All such geometric ruled surfaces
differ by elementary transforms. We want to choose one such that
$R^2$ is the smallest.

We mean by a curve $D$ on $S$ a nonzero effective divisor.
\begin{definition}
Let $D$ be an irreducible curve on a fibred surface $S$ with
fibration $f:S\to C$. If $f(D)$ is a point, we call $D$ a vertical
curve.
\end{definition}

\begin{lemma}[Lemma 6 in \cite{Xi91}]\label{chooselemma}
There is a birational morphism $\tilde\psi:\tilde P\to P$ over $C$,
where every fiber of the induced morphism $\varphi: P\to C$ is a
$\mP^1$, such that:

Let $\delta$ be the image of $\tilde \delta$ in $P$, and $R_h$ be
the sum of the non-vertical irreducible components of $R$. Then
$R^2$ is the smallest among all such choices, and the singularities
of $R_h$ are at most of order $g+1$. Therefore as $R$ is reduced,
the singularities of $R$ are of order at most $g+2$, and if $p$ is a
singular point of order $g+2$, $R$ contains the fiber of $\varphi$
passing through $p$.
\end{lemma}

\begin{definition}
Let $P$ be a geometric ruled surface over $C$, and $(R,\delta)$ be a
double cover datum on $P$. If $(R,\delta)$ satisfies that the
intersection number of $R$ with a general fiber $\Gamma$ of
$\varphi:P\to C$ is $R\Gamma=2g+2$, and the order of any singularity
of the non-vertical part $R_h$ of $R$ is at most $g+1$, we call
$(P,R,\delta)$ a {\it genus $g$ datum}.
\end{definition}
We have shown that there is a genus $g$ datum $(P,R,\delta)$ in
Lemma \ref{chooselemma} associated to a given hyperelliptic
fibration $f$ in the above. On the other hand, let $(P,R,\delta)$ be
a genus $g$ datum over a smooth curve $C$,  $\tilde\psi:\tilde P\to
P$ be the minimal even resolution of $(P,R)$, and let
$\tilde\theta:\tilde S\to \tilde P$ be the double cover determined
by $(\tilde R,\tilde\delta)$. Then $\tilde S$ is smooth. Let
$\rho:\tilde S\to S$ be the morphism of contracting all the vertical
$(-1)$-curves. Then we get a hyperelliptic fibration $f:S\to C$.

Hence we need to study the vertical $(-1)$-curves in $\tilde S$.

\begin{lemma}[\cite{Xi92}]\label{5.1.2}
Let $(P,R,\delta)$ be a genus $g$ datum, and $\Gamma$ be any fiber
of $P\to C$, whose inverse image in $\tilde S$ is a $(-1)$-curve. In
other words, the strict transform of $\Gamma$ in $\tilde P$ is a
$(-2)$-curve contained in $\tilde R$. If $g$ is even, then one of
the following two cases is satisfied,

(1). $R_h$ intersects with $\Gamma$ at two distinct points $x,y$,
$m_x(R_h)=m_y(R_h)=g+1$; or

(2). $R_h$ intersects with $\Gamma$ at one point, and the point is a
singularity of type $(g+1\to g+1)$, which is tangent to $\Gamma$.

If $g$ is odd, then $R_h$ intersects with $\Gamma$ at one point, and
it is a singularity of type $(g+2\to g+2)$, which is tangent to
$\Gamma$.
\end{lemma}

\begin{lemma}[\cite{Xi92}]\label{1.3}
 Suppose $E$ is a vertical $(-1)$-curve in $\tilde S$, then the image $\tilde E$ of $E$ in $\tilde P$ is
 an isolated $(-2)$-curve contained in $\tilde R$, and $\tilde E$ either comes from a
 blow-up of a singularity of $R$ with odd order, or is a strict transform of a fiber in Lemma \ref{5.1.2}.
 Conversely, for any singularity of $R$ with odd order or any fiber in Lemma \ref{5.1.2},
 there is a corresponding  vertical $(-2)$-curve.
\end{lemma}
The above two lemmas are easy (see \cite{Xi91}), and we omit their
proofs.
\begin{remark}
As stated in \cite{Xi91}, if we start from a hyperelliptic fibration
$f:S\to C$, we can choose a genus $g$ datum $(P,R,\delta)$ such that
$R^2$ is the smallest, and then the case (1) in Lemma \ref{5.1.2}
doesn't occur. Accordingly, Lemma \ref{1.3} turns to be Lemma 7 in
\cite{Xi91}. In what follows, we always assume that the genus $g$ datum
associated with $f$ satisfies that $R^2$ is the smallest.
\end{remark}

Consequently, in order to study hyperelliptic fibrations we only
need to consider genus $g$ data.

\subsection{Singularity indices}
Based on this preparation, we are able to define the singularity
indices.

Let $(P,R,\delta)$ be a genus $g$ datum over a smooth curve $C$, and
$\tilde\psi$ in expression (\ref{resolution}) be the minimal even
resolution of $(P,R)$. We decompose $\tilde\psi$ into $\psi':\tilde
P\to\hat P$ followed by $\hat\psi:\hat P\to P$, where $\psi'$ and
$\hat\psi$ are composed respectively of negligible and
non-negligible blowing-ups. We may assume
$\hat\psi=\psi_1\circ\cdots\circ\psi_t$, for $t\leq r$. And denote
by $(\hat R,\hat\delta)$ the reduced even inverse image of
$(R,\delta)$ in $\hat P$.

\begin{definition}\label{kktype}
Let $x_i$ be a singularity of $R_{i-1}$ of odd order $2k+1~(1\leq
k\leq [(g+1)/2])$. If $R_i$ has a unique singularity on the inverse
image of $x_i$, say $x_{i+1}$, with order $2k+2$, then we call $x_i$
a {\it singularity of type $(2k+1\to 2k+1)$}.
\end{definition}

\begin{definition}
Let $f:S\to C$ be a fibration and $D$ a reduced curve on $S$. Let
$\phi:D\to C$ be the natural projection induced by $f$. Let
$\nu:\tilde D\to D$ be the normalization of $D$, $D_h$ be the union
of all the irreducible components of $\tilde D$ which maps
projectively onto $C$, and $\nu_h:D_h\to D$ be the induced map. The
ramification index $r(D)$ of $\phi$ is defined as follows:

If $q\in D_h$ is a ramification point of $\phi\circ\nu_h$, then the
ramification index $r_q(D)$ is defined as usual;

If $p$ is a singularity of $D$ with order $m_p$, then the
ramification index is $r_p(D)=m_p(m_p-1)$;

If $E$ is an isolated vertical curve of $\tilde D$, we define the
ramification index to be $r_E(D)=\chi_{\mathrm{top}}(E)$;

Furthermore, we define
\begin{equation}
r(D):=\sum_{q\in D_h}r_q(D)+\sum_{p\in D}m_p(m_p-1)-\sum_{\substack{E\subset\tilde D~\mbox{\tiny{isolated }}\\
\mbox{\tiny{vertical curve}}}}\chi_{\mathrm{top}}(E).
\end{equation}
\end{definition}
\begin{remark}
 It is easy to see that
\begin{equation}
 r(D)=D^2+DK_{S/C},
\end{equation}
  from the adjoint formula
$K_SD+D^2=-2\chi(\CO(D))$ (see \cite{Xi92}).
\end{remark}
When we consider singular fiber $F$ of $f$, the singularities and
ramification points of branch locus are those over $f(F)$ without
confusion.
 \begin{definition}[\cite{Xi91,Xi92}]\label{defnsingularindex}
 Let $f:S\to C$ be a hyperelliptic fibration, and $(P,R,\delta)$ be the corresponding
 genus $g$ datum.
  Suppose $F$ is any fiber of $f$,
 we denote by $\Gamma$ the fiber of $P\to C$ over $f(F)$.
The singularity indices $s_k(F)~(2\leq k\leq g+2)$ are defined as
following.

 (1). Let $E_1,\ldots,E_k$
 be all the isolated vertical $(-2)$-curves in $\hat R$. Let $\hat R_p=\hat R-E_1-\cdots-E_k$,
 then $s_2(F)$ is defined to be the ramification index of $\hat R_p$ over the point
 $f(F)$. Concisely, if we denote by $\CR_{2,1}(F)$ the set of all ramification
 points of $\tilde R$ over $f(F)$, by $\CR_{2,2}(F)$ the set of all
 singularities of $\hat R_p$, and by $\CR_{2,-}(F)$ the set of all
 vertical components in $\hat R_p$, then
 \begin{equation}\label{s2eqn}
s_2(F)=\sum_{q\in\CR_{2,1}(F)}\big((\Gamma,\tilde
R)_q-1\big)+\sum_{q\in\CR_{2,2}(F)}m_q(m_q-1)-2|\CR_{2,-}(F)|.
 \end{equation}

 (2). If $k$ is odd, denote by $\CR_k(F)$ the set of all singularities of $R$
 of type  $(k\to k)$, then $s_k(F):=|\CR_{k}(F)|$.

 (3). If $k\geq4$ is even, denote by $\CR_{k}(F)$ the set of all singularities of $R$ of order $k$,
 not belonging to a singularity of type $(k+1\to k+1)$ or $(k-1\to
 k-1)$, then $s_k(F):=|\CR_{k}(F)|$.

 Define
\begin{equation}
s_k(f)=\sum_{i=1}^s s_k(F_i),
 \end{equation}
 where $F_1,\ldots, F_s$ are all the singular fibers of $f$.
 \end{definition}
\begin{remark}
Xiao introduced the singularity indices in order to compute the
contribution of singular fibers to the invariants $K_f^2, \chi_f$.
It is convenient to put $x_i,x_{i+1}$ in Definition \ref{kktype}
together, and regard the pair $\{x_i,x_{i+1}\}$ of points as one
singularity of type $(2k+1\to 2k+1)$, that is, the total
contribution of $x_i$ and $x_{i+1}$ to singularity indices adds one
to $s_{2k+1}$ only.
\end{remark}
\begin{example}\label{examplesingindex}
Let $(x,t)$ be the local coordinate of $\mP^1\times\Delta$, where
$\Delta$ is the open unit disc of $\mC$. Let
\begin{equation}
\begin{split}
h(x,t)=&(x+t)\big((x-a_0)^2+t\big)\big((x-a_1)^2+t^2\big)\\
&\cdot\big((x-a_2+t)^2+t^3\big)\big((x-a_2-t)^2+t^3\big)\big((x-a_3)^3+t^6\big),
\end{split}
\end{equation}
where $a_i$'s are distinct nonzero complex numbers. Let
$f:S_\Delta\to\Delta$ be the local hyperelliptic fibration of genus
$g$ defined by local equation
\begin{equation}
y^2=h(x,t).
\end{equation}
Let $F=f^{-1}(0)$ be the central fiber of $f$ over the origin,
$\Gamma$ the fiber of $\mP^1\times\Delta\to\Delta$ over the origin.
\vspace{-1cm}
\begin{center}
\setlength{\unitlength}{1mm}
\begin{picture}(50,53)(0,0)
\put(0,0){\makebox(0,0)[l]{Figure 3: The minimal even resolution}}
\multiput(10,10)(0,2){15}{\line(0,1){1}}
\put(5,15){\makebox(0,0)[l]{$\scriptstyle{p_3}$}}
\put(6,10){\makebox(0,0)[l]{$\scriptstyle{\Gamma}$}}
\put(8,15){\line(1,0){4}} \qbezier(8,16)(10,14.3)(12,16)
\qbezier(8,14)(10,15.7)(12,14)

\put(5,20){\makebox(0,0)[l]{$\scriptstyle{p_2}$}}
\qbezier(8,21)(10,19.1)(12,21) \qbezier(8,19)(10,20.9)(12,19)
\qbezier(9,20.5)(9.5,21)(10,20) \qbezier(10,20)(10.5,21)(11,20.5)
\put(8.5,21.5){\makebox(0,0)[l]{$\scriptscriptstyle{2}$}}
\put(10.5,21.5){\makebox(0,0)[l]{$\scriptscriptstyle{2}$}}
\put(5,25){\makebox(0,0)[l]{$\scriptstyle{p_1}$}}
\put(8,26){\line(2,-1){4}}  \put(8,24){\line(2,1){4}}

\put(5,30){\makebox(0,0)[l]{$\scriptstyle{p_0}$}}
\put(12.5,30){\oval(5,5)[l]} \qbezier(10,32)(10.5,32.5)(10.9,32)
\put(10.5,33.5){\makebox(0,0)[l]{$\scriptscriptstyle{2}$}}

\put(8,37){\line(1,0){4}}
\put(15,25){\makebox(0,0)[l]{$\stackrel{\tilde\psi}{\leftarrow}$}}
\put(23,10){\makebox(0,0)[l]{$\scriptstyle{\Gamma}$}}
\multiput(25,10)(0,2){15}{\line(0,1){1}}
\put(22,13){\makebox(0,0)[l]{$\scriptstyle{p_3}$}}
\put(32,13){\makebox(0,0)[l]{$\scriptstyle{p_{31}}$}}
\put(45,15){\makebox(0,0)[l]{$\scriptstyle{E_{31}}$}}
\put(39,8){\makebox(0,0)[l]{$\scriptstyle{E_{32}}$}}
\put(23,15){\line(1,0){20}} \multiput(38,16)(0,-2){5}{\line(0,1){1}}
\multiput(37,13)(0,-1){3}{\line(1,0){2}}

\put(22,18){\makebox(0,0)[l]{$\scriptstyle{p_2}$}}
\put(28,18.5){\makebox(0,0)[l]{$\scriptstyle{p_{21}}$}}
\put(36,18.5){\makebox(0,0)[l]{$\scriptstyle{p_{22}}$}}
\put(45,20){\makebox(0,0)[l]{$\scriptstyle{E_2}$}}
 \multiput(23,20)(2,0){11}{\line(1,0){1}}
 \qbezier(28,21.2)(30,19.2)(32,21.2)  \qbezier(31.5,20.5)(32,20.3)(31.5,20)
 \put(32.5,21){\makebox(0,0)[l]{$\scriptscriptstyle{2}$}}
\qbezier(36,21.2)(38,19.2)(40,21.2)
\qbezier(39.5,20.5)(40,20.3)(39.5,20)
 \put(40.5,21){\makebox(0,0)[l]{$\scriptscriptstyle{2}$}}
\put(22,23){\makebox(0,0)[l]{$\scriptstyle{p_1}$}}
\put(45,25){\makebox(0,0)[l]{$\scriptstyle{E_1}$}}
 \multiput(23,25)(2,0){11}{\line(1,0){1}}
 \put(30,24){\line(0,1){2}}  \put(32,24){\line(0,1){2}}

\put(22,30){\makebox(0,0)[l]{$\scriptstyle{p_0}$}}
\put(27.5,30){\oval(5,5)[l]} \qbezier(25,32)(25.5,32.5)(25.9,32)
\put(25.5,33.5){\makebox(0,0)[l]{$\scriptscriptstyle{2}$}}

\put(23,37){\line(1,0){4}}
\end{picture}
\end{center}

\vspace{1cm}

The branch locus is $R=\{(x,t)\in\mP^1\times\Delta:h(x,t)=0\}$, and
$R\Gamma'=12$, where $\Gamma'$ is any fiber of
$\mP^1\times\Delta\to\Delta$. Hence $g=5$ by Riemann-Hurwitz
formula. Let $p_i=(0,a_i)~(i=0,1,2,3)$ be the points of $R$, where
$p_2$ and $p_3$ are non-negligible.

Let $p_{21}$ and $p_{22}$ be the infinitely near points of $p_2$,
which are smooth points of $\hat R$. Let $p_{31}$ be the infinitely
near singularity of $p_3$, then $\{p_3,p_{31}\}$ is a singularity of
type $(3\to3)$. Therefore $\hat R_p=R$ which is the strict transform
in $\hat P$, and
\begin{equation}\label{eqr}
\begin{split}
&\CR_{2,1}(F)=\{p_0,p_{21},p_{22}\},~\CR_{2,2}(F)=\{p_1\},~\CR_{2,-}(F)=\emptyset,\\
&\CR_{3}(F)=\{\{p_3,p_{31}\}\}, ~~~\CR_4(F)=\{p_2\}.
\end{split}
\end{equation}
Furthermore, the singularity indices are
\begin{equation}\label{eqs}
\big(s_2(F),s_3(F),\ldots,s_7(F)\big)=(5,1,1,0,0,0).
\end{equation}
\end{example}

Using the singularity indices, Xiao obtained the following
local-global formula.

\begin{theorem}[Theorem 5.1.7, \cite{Xi92}]\label{Theorem 5.1.7}
Let $f:S\to C$ be a hyperelliptic fibration of genus $g$, then
\begin{equation}
\begin{split}
&(8g+4)\chi_f=g\big(s_2(f)-2s_{g+2}(f)\big)+\sum_{k=2}^{[\frac{g-1}2]}2(k+1)(g-k)s_{2k+2}(f)\\
&~~~~~~~~~~~~~~~+\sum_{k=1}^{[\frac{g+1}2]}4k(g-k)s_{2k+1}(f),\\
&e_f=
s_2(f)-3s_{g+2}(f)+\sum_{k=1}^{[\frac{g-1}2]}2s_{2k+2}(f)+\sum_{k=1}^{[\frac{g+1}2]}s_{2k+1}(f),\\
&(2g+1)K_f^2=(g-1)s_2(f)+3s_{g+2}(f)+\sum_{k=1}^{[\frac{g-1}2]}a_ks_{2k+2}(f)+\sum_{k=1}^{[\frac{g+1}2]}b_ks_{2k+1}(f),
\end{split}
\end{equation}
where $a_k=6\big((k+1)(g-k)-4g-2\big)$, and $b_k=12k(g-k)-2g-1$.
\end{theorem}

As a corollary, he proved that

\begin{corollary}[\cite{Xi92}]\label{coroXiao}
Suppose $f$ is hyperelliptic, then the slope of $f$
$$
\frac{4g-4}{g}\leqslant\lambda_f\leqslant\begin{cases} 12-
\frac{8g+4}{g^2}, & \mbox{if~} g \mbox{~is~even,}\\
12-\frac{8g+4}{g^2-1}, & \mbox{if~} g \mbox{~is~odd,}
\end{cases}
$$ Moreover, the left equality holds if and only if
$s_2(f)\neq0, s_k=0~ (k>2)$, and the right equality holds if and
only if $s_{2[g/2]+1}\neq0$ and other singularity indices are all
zero.
\end{corollary}

\section{Modular invariants in semistable case}
At the beginning of this section, we fix notations firstly.

Let $(P,R,\delta)$ be a genus $g$ datum over a smooth curve $C$, and
$\tilde\psi$ in (\ref{resolution}) be the minimal even resolution.
Let $f:S\to C$ be the fibration determined by the datum, and $F$ be
any singular fiber of $f$. Denote by $\tilde F$ the total transform
of $F$ by $\rho:\tilde S\to S$, which is a birational morphism
contracting all the vertical $(-1)$-curves. Let $\Gamma$ be the
fiber of $\varphi:P\to C$ over $t=f(F)$, and we call $\Gamma$ {\it
the image of $F$ in $P$} briefly. Let
$\tilde\Gamma=\tilde\psi^*(\Gamma)$ be the total transform of
$\Gamma$ by the minimal even resolution $\tilde\psi:\tilde P\to P$
of $R$. To keep it simple, we also denote by $R$~(resp. $\Gamma$)
the strict transform of $R$ (resp. $\Gamma$) under the even
resolution $\tilde\psi$.

\centerline{\mbox{\xymatrix{
\tilde F\subset\tilde{S} \ar@{->}"1,3"^-{\tilde{\theta}}\ar[d]_-{\rho} &  & {\tilde{P}\supset \tilde\Gamma} \ar[d]^-{\tilde\psi} \\
F\subset S  \ar@{-->}"2,3"^-{}  \ar[dr]_-f   &  & P\supset \Gamma \ar[dl]^-{\varphi}\\
&t\in C & }}}

Denote by $B=\tilde\theta^{-1}(\Gamma)$ the inverse image of
$\Gamma$ in $\tilde F$, and by $B_i=\tilde\theta^{-1}(E_i)$ the
inverse image of the exceptional curve $E_i$. Then $B$ (resp. $B_i$)
may be composed by two irreducible curves $B'$ and $B''$ (resp.
$B_{i}'$ and $B_{i}''$). Let
\begin{equation}\label{tildegamma}
\tilde\Gamma=\Gamma+\sum_{i=1}^rm_iE_i,
\end{equation}
then
\begin{equation}
\tilde
F=\tilde\theta^*(\tilde\Gamma)=\tilde\theta^*(\Gamma)+\sum_{i=1}^rm_i\tilde\theta^*(E_i)
=nB+\sum_{i=1}^rn_iB_i,
\end{equation}
where $n=1,2$ and $n_i=m_i$ or $n_i=2m_i$. Therefore, $F=\rho(\tilde
F)$ is obtained by contracting $(-1)$-curves in $\tilde F$.

\begin{definition}
An {\it even resolution at point $p$} of $R$  is a sequence of
blowing-ups
$\check\psi_p=\psi_1\circ\psi_2\circ\cdots\circ\psi_l:\check P\to P$
\begin{equation}\label{localresolution}
(\check P,\check R){=}(P_l,R_l)\stackrel{\psi_l}{\to} \cdots{\to}
(P_2,R_2)\stackrel{\psi_2}{\to} (P_1,R_1) \stackrel{\psi_1}{\to}
(P_0,R_0){=}(P,R),
\end{equation}
satisfying the following conditions:

(i). all the points of $\check R$ infinitely near $p$, including
$p$, are smooth,

(ii). $R_i$ is the reduced even inverse image of $R_{i-1}$ under
$\psi_i$.

Furthermore, $\check\psi_p$ is called the {\it minimal even
resolution at $p$} of $R$ if

(iii). $\psi_i$ is the blowing-up of $P_{i-1}$ centered at a
singular point $p_i$ of $R_{i-1}$ which is infinitely near $p$ for
any $1\leq i\leq l$.
\end{definition}

If the resolution $\check\psi_p$ is minimal, we call the desired
number $l$ of blowing-ups {\it the length} of the minimal even
resolution $\check\psi_p$ at $p$ of $R$, and we denote the length
$l$ by $l_p$. The exceptional curves $E_i$'s ($1\leq i\leq l$) in
$P_l$ are called {\it exceptional curves from $p$} briefly.

For example, if $p$ is an ordinary singularity of even order, then
$l_p=1$. If $p$ is a singularity of type $(3\to3)$, then $l_p\geq2$.

Let $p$ be any singularity of $R$, and $E_1,\ldots,E_{l_p}$ be all
the exceptional curves from $p$ in $P_{l_p}$. Set
\begin{equation}
\CE_p:=m_1E_1+\ldots+m_{l_p}E_{l_p},
~~~\CB_p:=\tilde\theta^*(\CE_p),
\end{equation}
where
$m_i=\mathrm{mult}_{E_i}(\CE_p)=\mathrm{mult}_{E_i}(\tilde\Gamma)$
(See (\ref{tildegamma})). Then we call $\CE_p$ {\it the block of
$\tilde\Gamma$ from $p$}, and call
\begin{equation}
F_p:=\rho(\CB_p).
\end{equation}
{\it the block of $F$ from $p$}.

Assume that $\Gamma$ is not contained in $R$. Let $p_1,\ldots,p_e$
be all the singularities of $R$ on $\Gamma$ in $P$, and
$\CB_{p_0}=\tilde\theta^*(\Gamma)$, then we can decompose $F$ into
finite blocks
\begin{equation}
 F=F_{p_0}+F_{p_1}+\ldots+F_{p_e},
\end{equation}
and we call it {\it the modular decomposition} of $F$.
\begin{example}\label{example}[Continuation of Example \ref{examplesingindex}]
Let $f:S_\Delta\to\Delta$ be the local fibration in Example
\ref{examplesingindex}. Then $l_{p_1}=1,~l_{p_2}=1,~l_{p_3}=2$. The
blocks of $\tilde\Gamma$ are
\begin{equation*}
\CE_{p_1}=E_1,~~\CE_{p_2}=E_2,~~\CE_{p_3}={E}_{31}+E_{32}.
\end{equation*}
\begin{center}
\begin{picture}(30,30)(0,0)
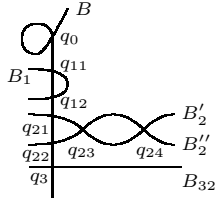
\put(-15,0){\makebox(0,0)[l]{Figure 4: Modular decomposition of
$F$}} \put(8,30){\makebox(0,0)[l]{$\scriptstyle{B}$}}
\put(6,26){\makebox(0,0)[l]{$\scriptstyle{q_0}$}}
\put(-1,21){\makebox(0,0)[l]{$\scriptstyle{B_1}$}}
\put(6,22.5){\makebox(0,0)[l]{$\scriptstyle{q_{11}}$}}
\put(6,17.5){\makebox(0,0)[l]{$\scriptstyle{q_{12}}$}}
\put(22,16){\makebox(0,0)[l]{$\scriptstyle{B_2'}$}}
\put(22,12){\makebox(0,0)[l]{$\scriptstyle{B_2''}$}}
\put(22,7){\makebox(0,0)[l]{$\scriptstyle{B_{32}}$}}
\put(1,14){\makebox(0,0)[l]{$\scriptstyle{q_{21}}$}}
\put(1,10.5){\makebox(0,0)[l]{$\scriptstyle{q_{22}}$}}
\put(2,7.5){\makebox(0,0)[l]{$\scriptstyle{q_{3}}$}}
\put(7,11){\makebox(0,0)[l]{$\scriptstyle{q_{23}}$}}
\put(16,11){\makebox(0,0)[l]{$\scriptstyle{q_{24}}$}}

\qbezier(5,26)(5,28)(3,28) \qbezier(3,28)(1,28)(1,26)
\qbezier(3,24)(4,24)(7,30) \qbezier(1,26)(1,24)(3,24)
\put(5,5){\line(0,1){21}} \qbezier(2,22)(7,22)(7,20)
\qbezier(2,18)(7,18)(7,20)

\qbezier(2,16)(7,16)(9,14) \qbezier(2,12)(7,12)(9,14)
\qbezier(13,16)(11,16)(9,14) \qbezier(13,12)(11,12)(9,14)

\qbezier(13,16)(15,16)(17,14) \qbezier(13,12)(15,12)(17,14)
\qbezier(21,16)(19,16)(17,14) \qbezier(21,12)(19,12)(17,14)

\put(2,9){\line(1,0){20}}
\end{picture}
\end{center}

Here $E_1,E_2,E_{32}$ are not contained in $\tilde R$, and $E_{31}$
is contained in $\tilde R$. Then the blocks of $F$ are
\begin{equation*}
F_{p_0}=B,~~F_{p_1}=B_{1},~~F_{p_2}=B_{2}'+B_{2}'',~~F_{p_3}=B_{32}.
\end{equation*}
In the above equation, $B$ is a rational curve with a node $q_0$;
$B_1$ is $\mP^1$ meets $B$ at two points $q_{11},q_{12}$;
$B_{2}',B_{2}''$ are both $\mP^1$ meeting with $B$ at
$q_{21},q_{22}$ respectively and with each other at two points
$q_{23},q_{24}$; and $B_{32}$ is a smooth elliptic curve meeting
with $B$ at $q_3$. Then $F$ is semistable, and the modular
decomposition of $F$ is
\begin{equation*}
F=\sum_{i=0}^2F_{p_i}=B+B_{1}+B_{2}'+B_{2}''+B_{32}.
\end{equation*}
\end{example}

\subsection{Semistable criterion}

There is a criterion for semistable hyperelliptic fiber given by Tu
\cite{Tu08}. We rewrite the result and proof here, for the reference
is in Chinese.

\begin{lemma}[\cite{Tu08}]\label{tu}
Suppose $F$ is a semistable fiber of a hyperelliptic fibration
$f:S\to C$. Then we have the following,

(1).  If $g$ is odd, then $\Gamma$ is not contained in $R$; if $g$
is even and $\Gamma$ is contained in $R$, then $\Gamma$ is the fiber
in Lemma \ref{5.1.2}.

(2).  Suppose $p$ is a smooth point of $R$ in $P$, then the
intersection number of $R$ with $\Gamma$ at $p$ is $(R,\Gamma)_p\leq
2$.

(3). If $p$ is a singularity of $R$ in $P$, then we have
$(R,\Gamma)_p=\mathrm{ord}_p(R)$.

(4). Let $q$ and $E$ be the same as above. If $\ord_q(R)=l$ is even,
then $E$ is not contained in the branch locus $\tilde R$. If
$\ord_q(R)=l$ is odd, then either $E$ is contained in $\tilde R$,
and thus $E$ is from a singularity of type $(k\to k)~(k \mbox{ is
odd}, ~~k\geq l)$; or $E$ is not contained in $\tilde R$, and thus
$q$ is a singularity of type $(k\to k)$.

(5). Let $q\in R_i$ (for some $i\geq1$) be an infinitely near
singularity, then there is exactly one exceptional curve $E_q$ in
$P_j$ passing through $q$, and $(R,E_q)_q=\mathrm{ord}_q(R)$.

(6). Let $q$ be an infinitely near smooth point, and $E$ be the same
as above, then $(R,E)_q\leq 2$, and $E$ is not contained in $\tilde
R$.
\end{lemma}
\begin{proof}
(1). Suppose $\Gamma\subseteq \tilde R$, then $B$ is a component of
$\tilde F$ with multiplicity 2, for $\pi^*(\Gamma)=2B$, furthermore,
$B^2=\Gamma^2/2$. If $\Gamma^2\leq-4$, then $B^2\leq -2$, hence $B$
is a multiple component in $\tilde F$ which can not be contracted,
contradicting with the assumption that $F$ is semistable. Thus we
get that if $\Gamma\subseteq \tilde R$, then $\Gamma$ is a
$(-2)$-curve in $\tilde P$.

By Lemma \ref{5.1.2}, we know that if $g$ is even then $\Gamma$ is
as in (1) in Lemma \ref{5.1.2}. And if $g$ is odd, then any
singularity of $R$ is of type $(g+2\to g+2)$, and we need twice
blow-up so that the intersection point of $\Gamma$ with the
exceptional curve is a smooth point of $R$. Hence there is a
$(-1)$-curve, say $E_2$, with multiplicity 2 in $\tilde \Gamma$.
It's easy to see that $E_2$ is not contained in $\tilde R$, and
$\pi^*(2E_2)=2B_2$ in $F$ is irreducible with $B_2^2\leq-2$,
therefore $B_2$ is an un-contractible multiple component in
semistable curve $\tilde F$, which is impossible. In a word, when
$g$ is odd, $\Gamma$ is not contained in $R$.

(2). For what follows, we assume that $\Gamma$ is not contained in
$R$ since (1). Let $n=(R,\Gamma)_p$, we take the local coordinate
$(x,t)$ of $p$ such that the local equations of $\Gamma$ and $R$ at
$p$ are $ t=0$ and $t+x^n=0$ respectively. Then the local equation
of $F$ in $S$ is $y^2-x^n=0$. If $n\geq 3$, it is a singularity of
type $A_{n-1}$ on $F$, and then $F$ is not semistable.

(3). Suppose not, thus $(R,\Gamma)_p>\ord_p(R)$. Let $\psi_1$ be a
blow-up at $p$, and $E_1$ be the exceptional curve. Then the
intersection point $p'$ of $\Gamma$ with $E_1$ is still on $R$. Let
$\psi_2$ be the successive blow-up centered at $p_1$, and $E_2$ the
exceptional curve. Then the total transform of $\Gamma$ by
$\psi_1\circ\psi_2$ is
\begin{equation*}
\tilde\Gamma_2=\Gamma+2E_2+E_1,
\end{equation*}
and $B_2$ is with multiplicity at least 2 in $F$. Hence $B_2$ must
be a $(-1)$-curve in $\tilde S$, $E_2$ a $(-2)$-curve in $\tilde R$,
and $p_1$ must be a singularity of type $(k\to k)$ ~($k$ is odd)
(Lemma \ref{5.1.2}). Furthermore, there is a singularity $p_2$ on
$E_2$ of order $k+1$. Let $\psi_3$ be the blow-up centered at $p_2$
with exceptional curve $E_3$. Then
\begin{equation*}
\tilde\Gamma_3=\Gamma+2E_3+2E_2+E_1,
\end{equation*}
$E_3$ is not contained in $\tilde R$, and $B_3$ is an un-contractile
multiple component in $\tilde F$.

(4). Suppose $\ord_q(R)$ is even and $E$ is contained in $\tilde R$,
then $\ord_q(R_i)$ is odd. Let $\psi: P_{i+1}\to P_i$ be the blow-up
centered $q$ with exceptional curve $E'$ lying in branch locus, then
the intersection point $E'\cap E$ is a singularity of $\tilde
R_{i+1}$, and $E^2\leq-2$ in $P_{i+1}$. Hence $E^2\leq-4$ in $\tilde
P$, and then $B$ is an un-contractile multiple component in $F$.
Consequently, we proved the first part of (4).

The second part of (4) is a direct corollary of Lemma \ref{1.3}.

(5). Suppose $E_1$ and $E_2$ are both through $q$. When $\ord_q(R)$
is even, then the exceptional curve $E_3$ of the blow-up at $q$ is
of multiplicity at least 2, and $E_3$ is not contained in $\tilde
R$. So $B_3$ is an un-contractile multiple component in $F$. When
$\ord_q(R)=k$ is odd, then $q$ should be of type $(k\to k)$, and
$E_3$ is contained in $\tilde R$ and of multiplicity at least 2.
Blowing up the infinitely near singularity $q'$ of $q$, then the
exceptional curve $E_4$ is not contained in $\tilde R$ of
multiplicity at least 2, which is impossible. The proof of the
second part of (5) is analogous to that of (3).

(6). The proof of the second part is the same as that of (1), and
the rest is the same as that of (2). We omit the detail.
\end{proof}
\begin{remark}
Let $F$ be a semistable fiber of $f$, and $p$ be a singularity of
$R$. Then there is exactly one curve $E_p$ passing through $p$
(Lemma \ref{tu} (3)-(5)). We call $E_p$ {\it the exceptional curve
through $p$}. Note that $E_p$ is either $\Gamma$ or an exceptional
curve.
\end{remark}
\begin{corollary}\label{g+2=0}
If $F$ is a semistable hyperelliptic fiber of genus $g$, then
$s_{g+2}(F)=0$.
\end{corollary}
\begin{proof}
From Lemma \ref{tu} (1), we know that if $g$ is odd, then
$s_{g+2}(F)=0$; if $g$ is even, then $\Gamma$ is the fiber in Lemma
\ref{5.1.2}, but it is impossible by Lemma \ref{tu} (3), and  then
$s_{g+2}(F)=0$.
\end{proof}

\subsection{Proof of Theorem \ref{mainthm}}

%
%
We first consider the effect of the smooth points of $R$ to the
arithmetic genus.
\begin{lemma}\label{nodelem}
Let $F$ be a semistable fiber of $f$. Assume that the image $\Gamma$
in $P$ of $F$ is not contained in $R$. Suppose all the intersection
points $p_1,\ldots,p_{k_1}$, $q_1,\ldots,q_{k_2}$ of $R$ on $\Gamma$
are smooth, where $(\Gamma,R)_{p_i}=2$ and $(\Gamma,R)_{q_j}=1$.

1). If $k_2\neq0$, then $F$ is an irreducible curve with $k_1$ nodes
corresponding to $p_i$, the geometric genus of $F$ is $[(k_2-1)/2]$,
and
\begin{equation*}
p_a(F)=[(\Gamma R-1)/2]=[(2k_1+k_2-1)/2]=k_1+[(k_2-1)/2].
\end{equation*}

2). If $k_2=0$, then $F$ is composed of two smooth rational curves
meeting with each other at $k_1$ distinct points, thus
\begin{equation*}
F=\tilde\Theta^*(\Gamma)+\sum_{i=1}^{k_1}\tilde\Theta^*(E_i)
=(B'+B'')+\sum_{i=1}^{k_1}B_i,
\end{equation*}
where every irreducible component is a smooth rational curve, $B_i$
meets $B'$ and $B''$ normally at one point respectively for each
$1\leq i\leq k_1$, and there is no other intersection. Hence
\begin{equation*}
p_a(F)=[(\Gamma R-1)/2]=k_1-1.
\end{equation*}
\end{lemma}
\begin{proof}
The proof is obvious, and we omit it.
\end{proof}
Then we consider the effect of the singularities.
\begin{lemma}\label{lem3.11}
Suppose $F$ is a semistable fiber of $f$, and the image $\Gamma$ in
$P$ of $F$ is not contained in $R$. If $p$ is a singularity of $R$
such that the exceptional curve $E_p$ through $p$ is not contained
in the branch locus. Then the arithmetic genus of the block of $F$
from $p$ is
 \begin{equation}\label{eqpa}
p_a(F_{p})=\Bigg[\frac{(E_p,R)_p-1}{2}\Bigg].
\end{equation}
\end{lemma}
 \begin{proof}
We use induction on  the length $l_p$ of the minimal even resolution
$\tilde\psi_p$ of $R$ at $p$. Note that
$\big[\big((2k+2)-1\big)/2\big]=\big[\big((2k+1)-1\big)/2\big]=k$,
and $\ord_p(R)=(R,E_p)_p$ for any singularity of $R$ on $E_p$  from
Lemma \ref{tu}. We may assume $E_p$ is $\Gamma$, since the proof for
exceptional curves is similar.

If $l_p=1$, then $\ord_p(R)=2k+2$ is even and $p$ is an ordinary
singularity. The exceptional curve $E_1$ from $p$ is not contained
in $R$, and $E_1$ meets $R$ in $P_1$ transversely at $2k+2$ distinct
points. Hence $\CE_p=E_1$, and $\CB_p=B_1$ with $p_a(B_1)=k$.

If $l_p=2$ and $\ord_p(R)=2k+2$ is even, then there is exactly one
infinitely near singularity $p_1$ of $R$ in $P_1$, which is an
ordinary singularity of even order, say $2k_2$. Hence $E_1,E_2$ are
not contained in $R$, $\CE_p= E_1+E_2$, and $\CE_p$ meets $R$ in
$P_2$ transversely at $2k+2$ distinct points. Let $E_1R=2k_1+2$,
then $k_1+k_2=k$. Thus $p_a(B_1)=k_1$, $p_a(B_2)=k_2-1$, and $B_1$
intersects with $B_2$ at two points transversely.
\begin{equation*}
p_a(F_p)=p_a(\CB_p)=p_a(B_1)+p_a(B_2)+1=k.
\end{equation*}

If $l_p=2$ and $\ord_p(R)=2k+1$ is odd, then $p$ is a singularity of
$(2k+1\to 2k+1)$. So $E_1$ is contained in $R$, $E_2$ is not
contained in $R$, and $\CE_p=E_1+E_2$, where $E_2$ meets $R_2$ in
$P_2$ transversely at $2k+2$ distinct points. It's easy to see that
$B_1$ is a $(-1)$-curve and $B_2$ is a smooth curve with genus $k$.
Hence
\begin{equation*}
p_a(F_p)=p_a(B_2)=k.
\end{equation*}

Assume that (\ref{eqpa}) holds for any non-negative integer $l<
l_p$. We want to prove (\ref{eqpa}) holds for $l_p$.

If $\mathrm{ord}_p(R)=2k+1$ is odd, let $\psi_1:P_1\to P$ be the blowing-up at
$p$. Then there is exactly one infinitely near singularity $q$ of
$R$ in $P_1$, and $(R,E_1)_q=2k+1,~\ord_q(R_1)=2k+2$. Let
$\psi_2:P_2\to P_1$ be the successive blowing-up at $q$. It is clear
that $E_1$ is contained in $R$, but $E_2$ is not.

Let $q_1,\ldots,q_\alpha$ be all the infinitely near singularities
of $q$ in $P_2$. Hence $l_{q_i}<l_p$ for $1\leq i\leq \alpha$.
Suppose $q_1,\ldots,q_\beta~(\beta\leq\alpha)$ are all the
singularities with even order. Let $(R,E_2)_{q_i}=2k_i+2$ for $1\leq
i\leq \beta$, and let $(R,E_2)_{q_j}=2k_j+1$ for $\beta+1\leq
j\leq\alpha$. Let the total intersection number of $R$ with $E_2$ at
all the smooth points of $R$ in $P_2$ be $(R,E_2)_{\mathrm{sm}}$.
Then
\begin{equation*}
\begin{split}
2k+2&=\sum_{i+1}^\beta (2k_i+2)+\sum_{j=\beta+1}^\alpha(2k_j+1)+(R,E_2)_{\mathrm{sm}}+1\\
&=2(k_1+\cdots+k_\beta)+2\beta+2(k_{\beta+1}+\cdots+k_{\alpha})+(\alpha-\beta)+(R,E_2)_{\mathrm{sm}}+1\\
&=2(k_1+\ldots+k_{\alpha})+(R,E_2)_{\mathrm{sm}}+(\alpha+\beta)+1.
\end{split}
\end{equation*}
Hence
\begin{equation}\label{l1}
k=\Big(\sum_{i=1}^\alpha
k_i\Big)+\frac{(R,E_2)_{\mathrm{sm}}+\alpha+\beta-1}2.
\end{equation}
It is easy to see that in $\tilde P$,
\begin{equation*}
\tilde R E_2=(R,E_2)_{\mathrm{sm}}+(\alpha-\beta)+1.
\end{equation*}
By Lemma \ref{nodelem},
\begin{equation}\label{l2}
p_a(B_2)=\Big[\frac{\big((R,E_2)_{\mathrm{sm}}+(\alpha-\beta)+1\big)-1}2\Big]
=\frac{(R,E_2)_{\mathrm{sm}}+\alpha-\beta-1}2.
\end{equation}
The block of $\tilde\Gamma$ from $p$ is
\begin{equation*}
\CE_p=E_1+E_2+\sum_{i=1}^{\alpha}\CE_{q_i}.
\end{equation*}
Combining the equations (\ref{l1}) and (\ref{l2}), then
\begin{equation}
\begin{split}
p_a(F_p)&=p_a(\CB_{p}-2B_1)=p_a(B_2)+\Big(\sum_{i=1}^\alpha p_a(F_{q_i})\Big)+\beta\\
&=\frac{(R,E_2)_{\mathrm{sm}}+\alpha-\beta-1}2+\Big(\sum_{i=1}^\alpha
k_i\Big)+\beta\\
&=k,
\end{split}
\end{equation}
where the block $F_{q_i}$ intersects with $B_2$ at two points, and
adds one to the arithmetic genus for each $1\leq i\leq\beta$. Here
we used the induction assumption.

If $\mathrm{ord}_p(R)=2k+2$ is even, take $\psi_1:P_1\to P$ the blow-up at $p$.
Let $q_1,\ldots,q_\alpha$ be all the infinitely near singularities
of $p$ on $p_1$. Then the rest of the proof is the same as the odd
case above.
 \end{proof}

Now we can prove the identities between singularity indices
(Definition \ref{defnsingularindex}) with modular invariants
$\delta_i(F),\xi_j(F)$ (see (\ref{xi0eqn}) - (\ref{delta0eqn})).

 \begin{theorem}
Let $f:S\to C$ be a semistable hyperelliptic fibration of genus $g$,
and $F$ be any singular fiber of $f$, then
\begin{equation}
s_{2k+1}(F)=\delta_k(F)~(k\geq1),~~ s_{2k+2}(F)=\xi_k(F)~(k\geq0).
\end{equation}
\end{theorem}
\begin{proof}

(1). Proof of $s_{2k+1}(F)=\delta_k(F),~k\geq1$.

We define a bijective map
\begin{equation}
\alpha_{2k+1}: \mathcal {R}_{2k+1}(F)\to \mathcal {N}_{2k+1}(F)
\end{equation}
between sets as follows:

If $p\in \CR_{2k+1}(F)$, then $E_p$ (the exceptional curve through
$p$) is not contained in $\tilde R$. Let
$\tilde\Gamma=\CE_p^c+\CE_p$, where $\CE_p$ is the block of
$\tilde\Gamma$ from $p$. Then the decomposition of $F$ is
$F=F_p^c+F_p$, where $p_a(F_p)=\Big[\big((R,E_p)_p-1\big)/2\Big]=k$,
and $F_p^c$ intersects with $F_p^c$ at a point, say $q$, which is a
node of type $k$. We define $\alpha_{2k+1}(p)=q\in\CN_{2k+1}(F)$, then
$\alpha_{2k+1}$ is well-defined.

 On the other hand, if $q\in\CN_{2k+1}(F)$, then $F$ consists of a
 genus $k$ curve $F_q$ and a genus $g-k$
curve $F_q^c$, and $F_q$ meets $F_q^c$ at $q$ transversely. Then $q$
is an isolated fixed point of the hyperelliptic involution $\sigma$.
Thus the inverse image of $q$ in $\tilde F$ under $\rho:\tilde S\to
S$ is a $(-1)$-curve $B$. Hence $\tilde\theta(B)$ is a $(-2)$-curve
contained in $\tilde R$, which is from a singularity, say $p$, of
type $(2k'+1\to 2k'+1)$ (cf. Lemma \ref{1.3} and (4) in Lemma
\ref{tu}). Since $\tilde\theta^*(\CE_{p})=\rho^*(F_q)$, the
arithmetic genus of the block of $F$ from $p$ is
\begin{equation*}
k'=p_a(\tilde\theta^*(\CE_{p}))=p_a(F_q)=k.
\end{equation*}
Thus $p\in\CR_{2k+1}(F)$, and $\alpha_{2k+1}(p)=q$. Hence it is
clear that $\alpha_{2k+1}$ is surjective and injective.

Therefore,
$$s_{2k+1}(F)=|\CR_{2k+1}(F)|=|\CN_{2k+1}(F)|=\delta_k(F).$$

(2). Similar proof of $s_{2k+2}(F)=\xi_k(F),~k\geq1$.

We define a bijective map
\begin{equation}
\alpha_{2k+2}: \mathcal {R}_{2k+2}(F)\to \mathcal {N}_{2k+2}(F)
\end{equation}
between sets as follows:

If $p\in\CR_{2k+2}(F)$, then $E_p$ is not contained in $\tilde R$,
and the exceptional curve $E_1$ of the blowing-up at $p$ is not in
$\tilde R$ either. Hence $\tilde\theta^{-1}(p)$ consists of two
points $q,\sigma(q)$.   Let $\tilde\Gamma=\CE_p+\CE_p^c$, then
$F=F_q+F_q^c$, $p_a(F_q)=k$ and $F_q$ meets $F_q^c$ at $q,\sigma(q)$
transversely. So the nodal pair $\{q,\sigma(q)\}\in\CN_{2k+2}(F)$.
Hence we are able to define $\alpha_{2k+1}(p)=\{q,\sigma(q)\}$.

On the other hand, if $\{q,\sigma(q)\}\in\CN_{2k+2}(F)$, then
$F=F_q+F_q^c$, where $p_a(F_q)=k$, and they intersect with each
other at two points $q,\sigma(q)$ transversely. We may assume that
$\tilde F=F_q+F_q^c$, which meet at $q$ and $\sigma(q)$. Then
$\tilde\theta(q)=\tilde\theta(\sigma(q))$, say $p$, is an
intersection point of two curves not in $\tilde R$. Hence we can
decompose $\tilde \Gamma$ as $\tilde \Gamma=\CE_p+\CE_p^c$, where
$R\CE_p=2k+2$ and $\CE_p$ meets $\CE_p^c$ at $p$ only. The curve
$\CE_p$ is from a singularity of order $\ord_p(R)=R\CE_q=2k+2$.
Therefore, $p$ is the inverse image of $\{q,\sigma(q)\}$ under
$\alpha_{2k+2}$, and $\alpha_{2k+2}$ is bijective.

So
\begin{equation*}
s_{2k+2}(F)=|\CR_{2k+2}(F)|=|\CN_{2k+2}(F)|=\xi_k(F).
\end{equation*}

(3). Proof of $s_2(F)=\xi_0(F)$.

If $E$ is a vertical components of $\hat R$, then
$B=\tilde\theta^*(E)$ is a multiple component of $\tilde F$, so $B$
is a $(-1)$-curve for $F$ is semistable, and then we know that $E$
is a $(-2)$-curve in $\hat R$.  Hence $\hat R_p$ is the strict
transform of $R$ in $\hat P$, and $|\CR_{2,-}(F)|=0$.

If $p\in\CR_{2,1}(F)$, then $p$ is a smooth point of $R$,
$(R,E_p)_p=2,r_p(R)=1$, and $\tilde\theta^{-1}(p)$ is an
$\alpha$-type node $q$. Conversely, each $\alpha$-type node $q$ is a
singularity $p$ of type $A_1$ whose local equation is $t+x^2=0$. So
we get a bijective map
\begin{equation}
\alpha_{2,1}:\CR_{2,1}(F)\to\CN_{2,1}(F).
\end{equation}

If $p\in\CR_{2,2}(F)$, then $p$ is an ordinary double point, and
$r_p(R)=2$. By the same discussion in (2), we can obtain a bijective
map
\begin{equation}
\alpha_{2,2}:\CR_{2,2}(F)\to\CN_{2,2}(F),
\end{equation}

Hence,  we have that
\begin{equation*}
s_2(F)=|\CR_{2,1}(F)|+2|\CR_{2,2}(F)|=|\CN_{2,1}(F)|+2|\CN_{2,2}(F)|=\xi_0(F).
\end{equation*}

\end{proof}

Theorem \ref{mainthm} is a corollary of the above theorem.

\begin{remark}
From the Corollary \ref{coroXiao} and Theorem \ref{mainthm}, we know
that a family $f:S\to C$ of hyperelliptic semistable curves with
lowest slope if and only if the image $[f]$ of $f$ by the moduli map
$J$ intersects with $\Xi_0$ only, and $f$ with highest slope if and
only if $[f]$ intersects with $\Delta_{[g/2]}$ only. See \cite{LT13}
for families with highest slope.
\end{remark}

\begin{example}[Continuation of Example \ref{example}]
From the analysis of the blocks of $F$ in Example \ref{example}, we
can easy to know that
\begin{equation}
p_a(F_{p_1})=1,~p_a(F_{p_2})=1,~p_a(F_{p_3})=1,
\end{equation}
and the sets of nodes are
\begin{equation}
\begin{split}
&\CN_{2,1}(F)=\{q_0,q_{23},q_{24}\},
~\CN_{2,2}(F)=\{(q_{11},q_{12})\},\\
&\CN_{3}(F)=\{q_3\},~\CN_4(F)=\{(q_{21},q_{22})\}.
\end{split}
\end{equation}
Hence the numbers of nodes on $F$ are
\begin{equation}
\begin{split}
&\big(\xi_0(F),\xi_1(F),\xi_2(F)\big)=(5,1,0),\\
&\big(\delta_1(F),\delta_2(F)\big)=(1,0).
\end{split}
\end{equation}
Comparing these two equations with (\ref{eqr}) and (\ref{eqs}), we
give an example for the above theorem.
\end{example}

{\bf Acknowledgement.} I would like to thank Professor Shengli Tan and Professor Jun Lu
for their interesting and fruitful discussions.   Thanks to Fei Ye
for reading the preliminary version of this paper and making many
valuable suggestions and discussions. Finally, I wish to thank the referees for their suggestions, which greatly improved the exposition of this paper.

\clearpage

\end{document}